\documentclass[12pt,reqno]{amsart}
\usepackage{hyperref}
\usepackage{amssymb,amsthm}
\usepackage{amsfonts,cmtiup,comment,stmaryrd}

\usepackage{mathrsfs,cmtiup}

\makeatletter
\def\blfootnote{\xdef\@thefnmark{}\@footnotetext}
\makeatother

\newtheorem{theorem}{Theorem}[section]
\newtheorem{lemma}[theorem]{Lemma}
\newtheorem{proposition}[theorem]{Proposition}
\newtheorem{corollary}[theorem]{Corollary}

\theoremstyle{definition}
\newtheorem*{notation}{Notation}

\newtheorem{remark}[theorem]{Remark}
\newtheorem*{definition*}{Definition}
\newtheorem*{ackn}{Acknowledgement}
\numberwithin{equation}{section}

\begin{document}

\title{On a problem of B.~Hartley about a small centralizer in finite and locally finite groups}

\author{Evgeny Khukhro}
\address{E. I. Khukhro: Charlotte Scott Research Centre for Algebra, University of Lincoln, U.K.}
\email{khukhro@yahoo.co.uk}


\begin{abstract} It is proved that if a finite group $G$ has an automorphism  of order $n$ with $m$ fixed points, then $G$ has a soluble subgroup whose index and Fitting height are  bounded in terms of $m$ and $n$. As a corollary, a problem of B.~Hartley is solved in the affirmative: if a locally finite group $G$ has an element with finite centralizer, then $G$ has a subgroup of finite index which has a finite normal series with locally nilpotent factors.
\end{abstract}

\maketitle

\section{Introduction}

Let $\varphi $ be an automorphism (or an element) of a finite group $G$. Studying the structure of the group $G$ depending on $\varphi $ and the  fixed-point subgroup $C_G(\varphi )$ is one of the most important and fruitful avenues in finite group theory. The celebrated  Brauer--Fowler theorem  \cite{bra-fow55} bounds the index of the soluble radical in terms of the order $|C_G(\varphi )|$ when $|\varphi |=2$. Thompson's theorem \cite{tho59} gives the nilpotency of $G$ when $\varphi $ is of prime order and acts fixed-point-freely, that is, $C_G(\varphi)=1$. These results lie in the foundations of the classification of finite simple groups. Using the classification, Hartley \cite{har92} generalized the Brauer--Fowler theorem to any order of $\varphi$:  the group  $G$ has a soluble subgroup of index bounded in terms of $|\varphi |$ and $|C_G(\varphi )|$.

When the group $G$ is soluble, further information is sought in the first
place in the form of bounds for the Fitting height of the group or of  a suitable `large' subgroup. Especially strong results on bounding the Fitting height have been obtained in the case of  soluble groups of coprime automorphisms $A\leqslant  {\rm Aut\,}G$. Thompson \cite{tho64} proved that if both groups $G$ and $A$ are soluble and have coprime orders, then the Fitting height of $G$ is bounded in terms of the Fitting height of $C_G(A)$ and the number $\alpha (A)$ of prime factors of $|A|$ with account for multiplicities. Later the bounds in Thompson's theorem were improved in numerous papers, like Kurzweil's \cite{kur}, with definitive results obtained by Turull \cite{tur84} and Hartley and Isaacs \cite{har-isa90} with linear bounds in terms of $\alpha (A)$ for the Fitting height of the group or of a `large' subgroup.  A rank analogue of the Hartley--Isaacs theorem was proved by Mazurov and Khukhro \cite{khu-maz06}.

The case of non-coprime orders of $G$ and $A\leqslant {\rm Aut\,}G$ is more difficult. Bell and Hartley~\cite{be-ha} constructed examples showing that for any non-nilpotent finite group $A$ there are soluble groups $G$  of  arbitrarily high Fitting height admitting $A$ as a fixed-point-free group of automorphisms.   But if $A$ is nilpotent  and $C_G(A)=1$,  then the Fitting height of $G$ is bounded in terms of $\alpha (A)$ by a special case of Dade's  theorem \cite{dad69}.
Unlike the aforementioned `linear' bounds in the coprime case, the bound in Dade's paper was exponential. Improving this bound to a linear one is a difficult problem even in the case of a cyclic fixed-point-free group of automorphisms $A=\langle\varphi\rangle$; it was tackled in some special cases by Ercan and G\"{u}lo\u{g}lu \cite{er12,er04,er08}. So far the best general result in this area is a quadratic bound $7\alpha (\langle\varphi\rangle)^2$ obtained by Jabara \cite{jab}.

 Jabara's proof in \cite{jab} used a result of Busetto and Jabara \cite{bu-ja} about a soluble finite group~$G$ factorized into three different Hall subgroups $G=AB=BC=AC$: then the Fitting height of  $G$ is bounded in terms of the Fitting heights of $A,B,C$. The latter result is used in the proof of the following main theorem of the present paper. Henceforth we write, say,  ``$(a,b,\dots)$-bounded'' to abbreviate ``bounded above by some function depending only on the parameters $a,b,\dots$".

\begin{theorem}\label{t-fin}
If a finite group  $G$ admits an automorphism $\varphi$ of order $n$ having $m=|C_G(\varphi )|$ fixed points,  then $G$ has a soluble subgroup whose index and  Fitting height are $(m,n)$-bounded.
\end{theorem}

In the proof, the reduction to the soluble case is given by Hartley' theorem \cite{har92} based on the classification of  finite simple groups. Earlier this kind of result for non-coprime automorphisms had been known for an automorphism whose  order is a product of two prime-powers \cite{hartley, khu15}, and this `biprimary' case is essentially used in the proof of Theorem~\ref{t-fin}. One of the difficulties in dealing with a non-coprime automorphism $\varphi$ is in the fact that there may not exist $\varphi$-invariant Sylow or Hall subgroups. In Jabara's paper~\cite{jab}, although for a non-coprime automorphism $\varphi$, the existence of such $\varphi$-invariant Hall subgroups was guaranteed by the condition that $\varphi$ was fixed-point-free. A crucial step in the proof of Theorem~\ref{t-fin} is a reduction to the situation where $\varphi$-invariant Hall subgroups do exist.

Because the order of the automorphism in Theorem~\ref{t-fin} is not assumed to be coprime to $|G|$, the result also applies to any element with given order of centralizer.

\begin{corollary}\label{c-fin}
If a finite group  $G$ contains an element $g$ with centralizer of order $m=|C_G(g )|$,  then $G$ has a soluble subgroup whose index and  Fitting height are $m$-bounded.
\end{corollary}

One may conjecture, by analogy with the coprime theorem of Hartley and Isaacs \cite{har-isa90}, that under the hypotheses of Theorem~\ref{t-fin} there should be a subgroup of $(m,n)$-bounded index whose Fitting height is bounded in terms of $|\varphi|$ alone, or even in terms of $\alpha(\langle\varphi\rangle)$. So far this result is only known for automorphisms of  prime-power order due to Hartley and Turau \cite{ha-tu}.

But importantly, Theorem~\ref{t-fin} is sufficient for giving the following affirmative answer to Hartley's problem, which was recorded in 1995 as Problem~13.8(a) in the ``Kourovka Notebook'' \cite{kour} by Belyaev.

\begin{corollary}\label{c-loc-fin}
If a locally finite group  $G$ has an element $g $ with centralizer of order $m=|C_G(g)|$,  then $G$ has a locally soluble subgroup of finite $m$-bounded index which has a normal series of finite $m$-bounded length with locally nilpotent factors.
\end{corollary}

This corollary follows by the standard inverse limit argument, after application of
Corollary~\ref{c-fin} to every finite subgroup of $G$ containing $g$.

\begin{remark}\label{r-1}
  Bounds for the Fitting height naturally reduce further studies of finite groups with (almost) fixed-point-free automorphisms to the case of nilpotent groups, where there are important open  questions of bounding the derived length, or the nilpotency class. For example, it is still unknown if the derived length of a nilpotent finite group $G$ with a fixed-point-free automorphism $\varphi$ is bounded in terms of $|\varphi|$, apart from the  cases of $\varphi$ of prime order due to Higman  \cite{hi} and Kreknin and Kostrikin \cite{kr, kr-ko}, or of order $4$ due to Kov\'acs \cite{kov}. The same applies to almost fixed-point-free automorphisms, where the nice results have been obtained only for $\varphi$ of prime order by Khukhro \cite{khu90}, or of order 4  by Khukhro and Makarenko \cite{mak-khu06}. In addition, definitive general results have been obtained in the study of almost fixed-point-free $p$-automorphisms of finite $p$-groups by Alperin  \cite{alp}, Jaikin-Zapirain \cite{jai, jai02}, Khukhro \cite{khu85,khu93}, Medvedev \cite{me}, Shalev  \cite{sha93}.
\end{remark}

\begin{remark}\label{r-2}
Automorphisms with restricted fixed-point subgroups have also been studied in the context of residually finite and profinite groups. As a consequence of Thompson's theorem \cite{tho59}, a profinite group  $G$ with a coprime fixed-point-free automorphism of prime order $p$ is pronilpotent, and moreover, $G$ is nilpotent of $p$-bounded class due to the theorems of Higman  \cite{hi} and Kreknin and Kostrikin \cite{kr, kr-ko}. Similarly, there are profinite analogues of the results on finite groups with a coprime automorphism of prime order, or order 4,  having  few fixed points,  which were mentioned in Remark~\ref{r-1}.
 Shalev \cite{sh98} proved that if a residually finite torsion group admits a finite 2-group $Q$  of automorphisms for which $C_G(Q)$  is finite, then $G$
 is locally finite.  Shumyatsky \cite{shu00}
proved that a torsion pro-$p$-group admitting a fixed-point-free automorphism of finite order has finite exponent.  Using, inter alia, Khukhro's theorem \cite{khu90} on almost fixed-point-free automorphism of prime order, Shalev \cite{sh94} proved that
if every centralizer of an element of a profinite group $G$  is either finite or has finite index, then $G$ is abelian-by-finite. It is also worth mentioning that the  proof of Conjecture~C  about pro-$p$-groups of finite coclass by Shalev and Zelmanov \cite{sh-ze} relied on Kreknin's theorem \cite{kr} on solubility of Lie rings admitting a fixed-point-free automorphism of finite order.
Other restrictions on fixed points of automorphisms of profinite groups include Engel-type conditions, and we mention a few recent results in this direction. Acciarri, Khukhro, and Shumyatsky \cite{aks} proved that if a profinite group $G$ admits a coprime automorphism $\varphi$ of prime order  such that all elements of $C_G(\varphi)$  are right Engel, then $G$  is locally nilpotent. Khukhro and Shumyatsky \cite{khu-shu22} proved that if a profinite group $G$  admits a coprime automorphism $\varphi$ of prime order such that each element of $C_G(\varphi)$ has a finite right Engel sink, then $G$ contains an open locally nilpotent subgroup.
Khukhro and Shumyatsky \cite{khu-shu22a} proved that if a profinite group $G$ admits an elementary abelian $q$-group of automorphisms $A$ of order $q^2$ such that for each $a\in A\setminus \{1\}$  every element of $C_G(a)$
 has a countable Engel sink, then $G$  has a finite normal subgroup $N$
 such that $G/N$
 is locally nilpotent.
\end{remark}

\begin{ackn}
The author is grateful to the anonymous referees for their careful reading and several helpful comments and suggestions. 
\end{ackn}

\section{Preliminaries}

In this section we recall some definitions and auxiliary facts about automorphisms. We also state for convenience two theorems that are used in the next section. Other well-known results of Dade, Hartley, Thompson, et al. that are used in the proofs are left as references to the literature.

If a group $A$ acts by automorphisms on a group $B$ we use the usual notation for  centralizers $C_B(A)=\{b\in B\mid b^a=b \text{ for all }a\in A\}$ and for  commutators $[b,a]=b^{-1}b^a$ regarding  $b\in B$, $a\in A$ as elements of the natural semidirect product $B\rtimes A$. The automorphism induced by an automorphism $\varphi  $ on the quotient by a normal $\varphi $-invariant subgroup is denoted by the same letter $\varphi $.  An automorphism $\varphi$ of a finite group $G$ is said to be coprime if  the orders of~$\varphi$ and $G$ are coprime: $(|G|,|\varphi|)=1$.

We now recall some well-known properties of arbitrary, not necessarily coprime, automorphisms of finite groups.

\begin{notation}
For an automorphism $\varphi$ of a finite group $G$, let $I_G(\varphi)$ denote the set of all commutators $[g,\varphi]=g^{-1}g^\varphi$, $g\in G$.
\end{notation}

Although at least the first parts of the next lemma are well-known, we reproduce their proofs for the benefit of the reader.

\begin{lemma}  \label{l-aut}
  Let $\varphi$ be an automorphism of a finite group $G$.
\begin{itemize}
  \item[\rm (a)]  The cardinality $|I_G(\varphi)|$ is equal to the index of $C_G(\varphi)$ in $G$.

  \item[\rm (b)] If $N$ is a normal $\varphi$-invariant subgroup of $G$, then $|C_{G/N}(\varphi)|\leqslant |C_G(\varphi)|$.

  \item[\rm (c)] If $N$ is a normal $\varphi$-invariant subgroup of $G$ such that
  $|C_{G/N}(\varphi)|= |C_G(\varphi)|$, then $N\subseteq I_G(\varphi)$.
\end{itemize}
\end{lemma}

\begin{proof}
  (a) It is easy to verify that the mapping $x\mapsto [x,\varphi]$ factors through to a one-to-one correspondence between right cosets of $C_G(\varphi)$ and elements of $I_G(\varphi)$.

  (b) Every element $[\bar x,\varphi]$ of $I_{G/N}(\varphi)$ is the  image of an element $[x,\varphi]$ of $I_{G}(\varphi)$ under the natural homomorphism $G\to G/N$, where $\bar x$ is the image of $x$. Therefore
  \begin{equation}\label{e-1}
    |I_G(\varphi)|\leqslant |N|\cdot |I_{G/N}(\varphi)|.
  \end{equation}
    Using part (a) we obtain
  \begin{equation}\label{e-2}
  \frac{|G|}{|C_G(\varphi)|}\leqslant |N|\cdot\frac{|G/N|}{|C_{G/N}(\varphi)|}=\frac{|G|}{|C_{G/N}(\varphi)|}.
  \end{equation}
  Hence, $|C_G(\varphi)|\geqslant |C_{G/N}(\varphi)|$.

  (c)  Clearly, for the equation $|C_G(\varphi)|= |C_{G/N}(\varphi)|$ to hold, we must have equalities instead of inequalities in \eqref{e-2} and \eqref{e-1}. The equation $|I_G(\varphi)|= |N|\cdot |I_{G/N}(\varphi)|$ holds exactly when every full inverse image of every element of $I_{G/N}(\varphi)$ under the natural homomorphism $G\to G/N$ consists only of elements of $I_G(\varphi)$. In particular, this applies to the inverse image of $1=[1,\varphi]$, which is $N$.
\end{proof}

For a finite group $H$, we denote by $\pi(H)$ the set of prime divisors of $|H|$.
Recall that if $G$ is a soluble finite group, then for every set of primes $\pi\subseteq\pi(G)$ there are Hall $\pi$-subgroups of $G$, and any two Hall $\pi$-subgroups are conjugate in $G$.

Recall that $F(G)$ denotes the Fitting subgroup of a finite group $G$, which is the largest normal nilpotent subgroup of $G$. The Fitting series of $G$ starts with  $F_1(G) = F (G)$, and by induction $F_{i+1}(G)$ is the inverse image of $F (G/F_i(G))$. If $G$ is soluble, then the least number $h$ such that $F_h(G) = G$ is called the Fitting height $h(G)$ of $G$.

The following theorem of Busetto and Jabara \cite{bu-ja} is instrumental in the proof of the main results.

\begin{theorem}[{\cite[Theorem~1.1]{bu-ja}}]
  \label{t-bj}
  Let $G$ be a finite group and let $\sigma, \tau, \nu$ be three different subsets of $\pi (G)$ such that
  $$
  \pi (G)=\sigma\cup \tau=\sigma\cup \nu= \tau\cup \nu.
   $$
   Let $ G_{\sigma}$, $ G_{\tau}$, $G_{\nu}$ be some Hall $\sigma$-, $\tau$-, $\nu$-subgroups of $G$. Then the Fitting heights of $G$, $ G_{\sigma}$, $ G_{\tau}$, $G_{\nu}$ satisfy the inequality
$$h(G)\leqslant h(G_{\sigma}) + h(G_{\tau} ) + h(G_{\nu}) - 2.$$
\end{theorem}

The following theorem was proved by Hartley in the unpublished manuscript \cite{hartley}, and independently by Khukhro~\cite{khu15}. This is of course the special case of Theorem~\ref{t-fin} for an automorphism of order a product of two prime-powers.

\begin{theorem}[{\cite[Theorem~2.4]{hartley}, \cite[Theorem~1.4]{khu15}}]\label{t-kh}
If a finite group  $G$ admits an automorphism $\varphi$ of order $p^aq^b$ for some primes $p,q$ and non-negative integers $a,b$, then $G$ has a soluble subgroup whose index and  Fitting height are bounded above in terms of $p^aq^b$ and $|C_G(\varphi )|$.
\end{theorem}

\section{Proofs of the main results}\label{s-finite}

We do not try to optimize the functions bounding the index and Fitting height in Theorem~\ref{t-fin}. If such an optimization were to be attempted, we mention in parentheses the relevant results that would lead to sharper bounds in terms of $m$ and $n$.

\begin{proof}[Proof of Theorem~\ref{t-fin}]
Recall that $G$ is a  finite group  admitting an automorphism  $\varphi$ of order~$n$ having $m=|C_G(\varphi )|$ fixed points; we need to prove that $G$ has a soluble subgroup whose index and Fitting height are $(m,n)$-bounded. By Hartley's theorem \cite{har92}, the group~$G$ has a $\varphi$-invariant soluble subgroup $H$ of $(m,n)$-bounded index. Replacing $G$ with $H$, we can assume from the outset that $G$ is a soluble group, and we need to prove that its Fitting height is $(m,n)$-bounded. 

It is expedient to prove simultaneously a technical fact, which is the second part of the following proposition. The first part of this proposition is the statement of Theorem~\ref{t-fin} in the soluble case, so  the proof of Theorem~\ref{t-fin} will be complete with the proof of this proposition.

\begin{proposition}\label{p-p}
Let $G$ be  a  finite soluble group  admitting an automorphism  $\varphi$ of order~$n$ having $m=|C_G(\varphi )|$ fixed points.
\begin{itemize}
  \item[\rm (a)]  Then the Fitting height of $G$ is $(m,n)$-bounded,
  that is, $h(G)\leqslant f(m,n)$ for some function $f(m,n)$ depending only on $m$ and $n$.
  \item[\rm (b)] Suppose that $N$ is a normal $\varphi$-invariant subgroup of $G$ such that $$|C_{G/N}(\varphi)|= |C_G(\varphi)|.$$ Then the Fitting height of $N$ is $(m,n)$-bounded, that is, $h(N)\leqslant g(m,n)$ for some function $g(m,n)$ depending only on $m$ and $n$.
\end{itemize}
\end{proposition}

\begin{proof} We prove parts (a) and (b) simultaneously by induction on $m+n=|C_G(\varphi)|+|\varphi|$ establishing the existence of the required functions $f(m,n)$  and $g(m,n)$ bounding above the Fitting heights of $G$ and $N$. These  functions can be assumed to be monotonically non-decreasing with respect to each of the parameters $m,n$, by passing respectively to
$$
\max \{f(k,l)\mid k\leqslant m,\;l\leqslant n\}\qquad \text{and}\qquad \max \{g(k,l)\mid k\leqslant m,\;l\leqslant n\}.
$$
 
As a base of this induction, we can take the case $m=1$ (for any value of $n$), when the automorphism $\varphi$ has no non-trivial fixed points in $G$. Then the Fitting height $h(G)$ of $G$ is $n$-bounded by a special case of Dade's theorem \cite{dad69}, and the Fitting height $h(N)$ of $N$ is of course at most $h(G)$. Actually, in this fixed-point-free case Dade's theorem \cite{dad69} even gives  an upper  bound for $h(G)$ in terms of the composition length $\alpha(\langle\varphi\rangle)$ of $\langle\varphi\rangle$, and Jabara's result \cite{jab} improves Dade's exponential bound to the quadratic one $7\alpha(\langle\varphi\rangle)^2$. Thus we can put $f(1,n)=g(1,n)=7\alpha(\langle\varphi\rangle)^2$.

In the induction step, for given $m,n$, we firstly prove part (b) using the induction hypothesis for part (a), and then derive part (a).  

First we establish the existence of $\varphi$-invariant Hall subgroups of $N$ under the hypotheses of part~(b).

\begin{lemma}\label{l-hall}
Under the hypotheses of part (b) of Proposition~\ref{p-p}, for every set of primes $\pi\subseteq\pi (N)$ the group $N$ has a $\varphi$-invariant Hall $\pi$-subgroup.
\end{lemma}

\begin{proof}
By Lemma~\ref{l-aut}(c) the condition
$|C_{G/N}(\varphi)|= |C_G(\varphi)|$ implies that $N\subseteq I_G(\varphi)$. 
 Let $S$ be a Hall $\pi$-subgroup of $N$. If $S^\varphi\ne S$, then $S^\varphi=S^a$ for some $a\in N$. Since $N\subseteq I_G(\varphi)$, there is $b\in G$ such that $a=[b,\varphi]=b^{-1}b^\varphi$. Then $S^{b^{-1}}$ is also a Hall $\pi$-subgroup of $N$, which is $\varphi$-invariant:
$$
\big(S^{b^{-1}}\big)^\varphi=S^{a(b^{-1})^\varphi}=S^{b^{-1}b^\varphi(b^{-1})^\varphi}=S^{b^{-1}}.
\eqno\qedhere$$
\end{proof}

We return to proving part (b) of the proposition. If  $n=|\varphi|$  is a product of at most two prime powers, then the Fitting height $h(G)$ of $G$ is $(m,n)$-bounded by Theorem~\ref{t-kh}, and  the Fitting height $h(N)$ of $N$ is at most $h(G)$. (When $n$ is a prime-power, even a stronger conclusion holds by a result of Hartley and Turau \cite{ha-tu} giving a subgroup of $(m,n)$-bounded index whose Fitting height is at most the composition length $\alpha(\langle\varphi\rangle)$.)
Therefore we can assume that $n$ is divisible by at least three different primes. 

For any prime divisor $p$ of $n=|\varphi |$, let  $p^{k_p}$ be the highest power of $p$ dividing $n$,  let $\langle\varphi_p\rangle$ be the Sylow $p$-subgroup of $\langle\varphi\rangle$, and let $\langle\varphi_{p'}\rangle$ be the Hall $p'$-subgroup of $\langle\varphi\rangle$, so that $\varphi=\varphi_p\cdot \varphi_{p'}$. Let $N_{p'}$ be a $\varphi$-invariant Hall $p'$-subgroup of $N$, which exists by Lemma~\ref{l-hall}. Here, $N_{p'}$ is a Hall $\pi$-subgroup for $\pi=\pi(N)\setminus\{p\}$; it is possible that $N_{p'}=N$, when $p$ does not divide $|N|$.

\begin{lemma}
\label{l-hp}
The Fitting height of $N_{p'}$ is bounded in terms of $m$ and $n$.
\end{lemma}

\begin{proof}
Note that $\varphi_p$ is a coprime automorphism of $N_{p'}$ (and this is also true in the case where $N_{p'}=N$, that is, when $p$ does not divide $|N|$). The centralizer $C=C_{N_{p'}}(\varphi_p)$ admits the automorphism $\varphi_{p'}$ of order $n/p^{k_p}$ whose fixed-point subgroup $C_{C}(\varphi_{p'})$ is contained in $C_G(\varphi)$ and therefore has order at most $m$. Since $$|C_{C}(\varphi_{p'})|+|\varphi_{p'}|\leqslant m+n/p^{k_p}<m+n,$$
we can apply the induction hypothesis for part (a)  to $C$ and its automorphism~$\varphi_{p'}$, by which the Fitting height of $C=C_{N_{p'}}(\varphi_p)$ is at most $f(m,n/p^{k_p})$. By Thompson's theorem~\cite{tho64} applied to $N_{p'}$ and its coprime automorphism $\varphi_p$, the Fitting height $h(N_{p'})$ of $N_{p'}$ is bounded in terms of $|\varphi_p|=p^k<n$ and  the Fitting height of $C_{N_{p'}}(\varphi_p)$, which is  at most $f(m,n/p^{k_p})$. Hence the Fitting height $h(N_{p'})$ is $(m,n)$-bounded. (Actually, Thompson's theorem \cite{tho64} even provides a bound depending on the composition length, rather than the order, of the group of automorphisms, and  the later theorems of Turull \cite{tur84} or Hartley--Isaacs \cite{har-isa90} give better bounds.)
\end{proof}

We return to the proof of part (b) of the proposition. 
Let $g_p(m,n)$ be the bound for the Fitting height $h(N_{p'})$ of  $N_{p'}$ given by Lemma~\ref{l-hp}.

Recall that $n$ is divisible by at least three different primes by our assumption; so we pick any three different primes $p,q,r$ dividing $n=|\varphi|$.  By Lemma~\ref{l-hall} there exist $\varphi$-invariant Hall $p'$-, $q'$-, $r'$-subgroups $N_{p'},N_{q'},N_{r'}$, for which we obviously have
$$
N=N_{p'}\cdot N_{q'}=N_{p'}\cdot N_{r'}=N_{q'}\cdot N_{r'}.
$$
By the Busetto--Jabara Theorem~\ref{t-bj}, the Fitting height $h(N)$ of $N$ satisfies the inequality
$$h(N)\leqslant  h(N_{p'})+h(N_{q'})+h(N_{r'})-2.$$
Since by Lemma~\ref{l-hp} we have
$$
h(N_{p'})\leqslant g_p(m,n),\qquad  h(N_{q'})\leqslant g_q(m,n),\qquad h(N_{r'})\leqslant g_r(m,n),
$$
 we  complete the proof of part (b) by setting 
 $$
 g(m,n)=g_p(m,n)+ g_q(m,n)+ g_r(m,n)-2.
 $$

We now finish  the induction step in the proof of part (a) of the proposition; recall that the induction is conducted with respect to $m+n$. For the function $g(m,n)$ in part (b), if  $F_{g(m,n)}(G)\ne G$, then $F_{g(m,n)+1}(G)$ has Fitting height greater than $g(m,n)$ and therefore $|C_{G/F_{g(m,n)+1}(G)}(\varphi)|$ must be strictly smaller than $m=|C_G(\varphi)|$ by part (b). Clearly, the order of the automorphism of $G/F_{g(m,n)+1}(G)$ induced by $\varphi$ is at most $n$. By the induction hypothesis, the Fitting height of $G/F_{g(m,n)+1}(G)$ is at most $f(m-1,n)$. As a result, the Fitting height of $G$ is at most $g(m,n)+1+f(m-1,n)$.  In the remaining case where  $F_{g(m,n)}(G)=G$, the Fitting height of $G$ is at most $g(m,n)$.  Thus we can complete  the induction step for part (a) by setting   $f(m,n)=g(m,n)+1+f(m-1,n)$.
 \end{proof}
 As explained at the beginning of the section, Proposition~\ref{p-p} completes the proof of Theorem~\ref{t-fin}.
 \end{proof}

 \begin{proof}[Proof of Corollary~\ref{c-fin}]
 Recall that here we have a finite group  $G$ containing an element $g$ with centralizer of order $m=|C_G(g )|$; we need to show that  then $G$ has a soluble subgroup whose index and  Fitting height are $m$-bounded. The result follows immediately by Theorem~\ref{t-fin} applied to the inner automorphism of $G$ induced by conjugation by $g$: the order of this automorphism divides $m$, while the number of fixed points is exactly $m$.
 \end{proof}
 
 \begin{proof}[Proof of Corollary~\ref{c-loc-fin}]
 Recall that $G$ is a locally finite group containing an element $g$ with centralizer of order $m=|C_G(g)|$; we need to show that $G$ has a locally soluble subgroup of finite $m$-bounded index which has a normal series of finite $m$-bounded length with locally nilpotent factors.
We apply the well-known inverse
limit argument. Let $f(m)$ and $g(m)$ be the functions given by Corollary~\ref{c-fin} such that if $G$ is a finite group, then it has a soluble subgroup of index at most $f(m)$ whose Fitting height is at most $g(m)$. Let $\Sigma$ be the family of all finite subgroups
of $G$ containing $g$; this is a directed poset by inclusion. For
each $H\in \Sigma$ let $S_H$ be the set of all
subgroups $R\leqslant H$ of index at most $f(m)$ whose Fitting height is at most $g(m)$. Each $S_H$ is
non-empty by Corollary~\ref{c-fin}. Clearly, each $S_H$ is finite. For any two
$H_1,H_2\in \Sigma$ such that $H_1\geqslant H_2$ there is a map
$\varphi _{H_1,H_2}:S_{H_1}\rightarrow S_{H_2}$ given by taking
the intersections with~$H_2$. This inverse system of finite sets
has non-empty inverse limit (see, for example, \cite[Theorem~1.K.1]{kw}). The union of the corresponding subgroups $R$ over any element of the inverse limit gives the required subgroup of~$G$.
 \end{proof}

\begin{remark}
In this paper, we did not try to obtain explicit estimates  in terms of $m$ and $n$ for the functions bounding  the index and Fitting height in Theorem~\ref{t-fin}  and the corollaries. It was  important for us  that these functions depend only on $m$ and $n$, which fact makes it possible to derive Corollary~\ref{c-loc-fin}  for locally finite groups. If explicit estimates were to be attempted, one would have to begin with getting an explicit  bound for the index of a soluble subgroup in Hartley's generalized Brauer--Fowler theorem \cite{har92}, as well as a bound for the Fitting height of a finite soluble group with an automorphism of order $n=p^aq^b$ having $m$ fixed points implied by the theorems of Hartley \cite{hartley} or Khukhro  \cite{khu15}. Once these initial estimates are obtained, it would be straightforward to estimate the bounding functions in this paper by following the proofs in this paper. However, the more important direction of further research is the question of existence, under the hypotheses of Theorem~\ref{t-fin}, of a  subgroup of $(m,n)$-bounded index whose Fitting height is bounded in terms of $n=|\varphi|$ alone, or even in terms of the composition length $\alpha(\langle\varphi\rangle)$ of $\langle\varphi\rangle$.
\end{remark}

\end{document}